\documentclass[runningheads,a4paper]{llncs}
\usepackage[T1]{fontenc} % for llncs
\usepackage{amsfonts, amsmath, amssymb}
\usepackage[dvipsnames]{xcolor} 
\usepackage[hidelinks]{hyperref}
\hypersetup{linktocpage,bookmarksopen = true, bookmarksopenlevel = 1, colorlinks=true,
linkcolor = Bittersweet, citecolor = RoyalBlue , urlcolor = gray, pdfstartview = FitR}

\newcommand{\QQ}{\mathbb{Q}}
\newcommand{\ZZ}{\mathbb{Z}}

\newcommand{\FF}{\mathbb{F}}
\newcommand{\GL}{\operatorname{GL}}

\newcommand{\End}{\operatorname{End}}

\newcommand{\Pic}{\operatorname{Pic}}

\renewcommand{\div}{\operatorname{div}}

\begin{document}

\title{Extending class group action attacks via sesquilinear pairings}

\author{Joseph Macula\inst{1}\orcidID{0009-0004-1843-6459}\\and
Katherine E. Stange\inst{1}\orcidID{0000-0003-2294-0397}}

\authorrunning{Macula, J. and Stange, K.E.}

\institute{University of Colorado Boulder, Boulder, USA \\
\email{Joseph.Macula@colorado.edu,kstange@math.colorado.edu}}

\maketitle             

\begin{abstract}
	We introduce a new tool for the study of isogeny-based cryptography, namely pairings which are sesquilinear (conjugate linear) with respect to the $\mathcal{O}$-module structure of an elliptic curve with CM by an imaginary quadratic field $\mathcal{O}$.  We use these pairings to study the security of problems based on the class group action on collections of oriented ordinary or supersingular elliptic curves.  This extends work of \cite{attacks} and \cite{Level}.
\keywords{Isogeny-bas/Users/josephmacula/Desktop/untitled foldered cryptography  \and Pairings \and Elliptic Curves.}
\end{abstract}

\section{Introduction}

The use of isogeny graphs in cryptography dates to \cite{CGL,Couveignes,RS}.  The latter proposals were for public-key cryptography based on an ordinary isogeny graph.  In particular, the class group $\operatorname{Cl}(\mathcal{O})$ of an order $\mathcal{O}$ in an imaginary quadratic field $K$ acts on the set of ordinary elliptic curves over $\overline{\FF}_p$ with CM by $\mathcal{O}$.  For efficiency, CSIDH was proposed \cite{CSIDH}, making use of supersingular curves with an action by the class group of the Frobenius field.  More recently, this was generalized to OSDIH \cite{OSIDH}, making use of other imaginary quadratic fields in the endomorphism algebra.  Recently, SIDH adaptations based on related ideas have been proposed \cite{NewCounter}.  Our paper concerns \emph{oriented} elliptic curves, which refers to attaching the data of an embedding of a particular imaginary quadratic order $\mathcal{O}$ into the endomorphism ring.  All these public-key proposals are examples of class group actions on oriented curves.

The security of these schemes relies on variants of the Diffie-Hellman problem for the class group action.  The security of these problems has drawn a great deal of interest, and not all instances of the problem have so far proven to be secure.  If the class group is even, the decisional Diffie-Hellman problem is broken by the use of genus theory \cite{DDH1,DDH2}.  These papers make use of the Weil and Tate pairings to compute certain associated characters.  More recently, \cite{attacks} makes use of generalizations of Weil and Tate pairings to break certain instances of the class group action problem (i.e., determining which class group element takes one given oriented curve to another) when the discriminant has a large smooth square factor and the degree is known.  Pairings have also appeared in the study of oriented elliptic curves in \cite{pairing-volcano}, to navigate the isogeny graph.  For other interactions between pairings and isogeny-based cryptography, see \cite{PairingGroups,Effective}. 

The attacks in \cite{attacks} use pairings to reduce a hidden isogeny problem with known degree for the class group action to the SIDH problem recently broken using higher dimensional abelian varieties \cite{SIDHbreak1,SIDHbreak2,SIDHbreak3}.  In short, if the degree of a secret isogeny $\phi: E \rightarrow E'$ is known, and it is known that $\phi P \in \ZZ P'$ for $P \in E$ and $P' \in E'$, then we can make use of a relationship of the form
\[
	\langle P, P \rangle^{\deg \phi} = \langle \phi P, \phi P \rangle =  \langle k P', k P' \rangle = \langle P', P' \rangle^{k^2}
\]
by solving a discrete logarithm to obtain the relationship $k^2 \equiv \deg \phi \pmod{m}$, and thereby solve for $k$.  With this, we (essentially) obtain the image $\phi P$ of $P$, which is the type of information provided in the SIDH problem.  The classical SIDH problem (for which we now have efficient methods) requires the image of two basis points, and this provides only one.  To close the gap, \cite{attacks} uses results of \cite{Level} which reduce SIDH$_1$, in which the image of only one torsion point is provided, to classical SIDH, provided the order of the point is square.  More recent work presented but not yet available \cite{unpubCastryck} uses pairings to generalize the SIDH attacks so that torsion images of any sufficiently large subgroup suffice.

These attacks require that the degree of the secret isogeny is known.  This is the case in constant-time implementations aimed at preventing side-channel attacks such as those in \cite{Side}; see \cite{attacks} for more details.  Furthermore, in \cite[Lemma 14]{Level}, the authors give a heuristic reduction from the group action 
 problem to the same problem with known degree.  \textbf{In this paper we will assume throughout that the degree of the secret isogeny is known.}

In this paper we introduce a new tool for understanding these results and pushing such attacks further.  In \cite{katepairing}, certain new generalized pairings $\widehat{W}$ and $\widehat{T}$ (generalizing the usual Weil and Tate pairings) are defined, which are $\mathcal{O}$-sesquilinear, meaning that
\[
	\langle \alpha x, \beta y \rangle = \langle x, y \rangle^{\overline{\alpha}\beta}
\]
for $\alpha, \beta \in \mathcal{O}$.  In particular, they take values in an $\mathcal{O}$-module formed by extending scalars from the usual domain $\FF_q^*$.  

In particular, we need now assume only that $\phi P \in \mathcal{O}P'$ and obtain a relationship
\[
	\langle P, P \rangle^{\deg \phi} = \langle \phi P, \phi P \rangle =  \langle \lambda P', \lambda P' \rangle = \langle P', P' \rangle^{N(\lambda)},
\]
where $\lambda \in \mathcal{O}$.  The new pairings are amenable to a Miller-type effective algorithm for their computation, and carry all the useful properties of the Weil and Tate pairings, especially compatibility with $\mathcal{O}$-oriented isogenies.

The paper \cite{attacks} provides a taxonomy of known generalized pairings, but all of these are only $\ZZ$-bilinear with image in $\FF_q^*$.

One important difference of these sesquilinear pairings from the generalized pairings previously considered is their non-degeneracy.  In \cite{attacks}, there is a classification theorem for \emph{cyclic self-pairings compatible with oriented endomorphisms}. These are functions $f_m : C \rightarrow \mu_m$ where $C$ is a cyclic subgroup of $E[m]$ whose image under $f_m$ spans $\mu_m$, with the following properties: $f(\lambda P) = f(P)^{\lambda^2}$, $\iota(\sigma)(P) \in C$, and $f(\iota(\sigma)P) = f(P)^{N(\iota(\sigma))}$ for $\iota$ an orientation of a given imaginary quadratic order $\mathcal{O}$, $\sigma \in \mathcal{O}$, and $P \in C$. They essentially show that such pairings can only be non-trivial for $m$ dividing the discriminant $\Delta_\mathcal{O}$ of $\mathcal{O}$.

The requirement that $m$ divide $\Delta_{\mathcal{O}}$ limits the applicability of their attacks on the class group action to situations where the discriminant has a good factorization.
 We demonstrate that by extending to $\mathcal{O}$-sesquilinear pairings, whose domain is not $\ZZ$-cyclic but instead $\mathcal{O}$-cyclic, we obtain many more non-trivial self-pairings to work with.

 The use of these new $\mathcal{O}$-sesquilinear pairings offers several clarifying conceptual advantages, and partially answers several of the open problems posed in \cite{attacks}.  However, they are not a magic bullet:  we show (Theorem~\ref{thm:equiv}) that the computation of these pairings is essentially equivalent to the computation of the $\mathcal{O}$-orientation, provided discrete logarithms are efficient in $\mu_m$ (for example, if $m$ is smooth).
 \vspace{1em}

\textbf{Conceptual contributions.}
\begin{enumerate}
\item We introduce the new $\mathcal{O}$-sesquilinear pairing $\widehat{T}$ in the cryptographic context.
\item We show that these pairings give rise to many non-degenerate $\mathcal{O}$-cyclic self-pairings, without a requirement that $m$ divide the discriminant (Theorem~\ref{thm:order}).
\item We characterize elliptic curves for which $E[m]$ is a cyclic $\mathcal{O}$-module (Theorem~\ref{thm:cyclic}):  $E[m]$ is $\mathcal{O}$-cyclic if and only if the $\mathcal{O}$-orientation is $m$-primitive. 
\item We show an equivalence between computation of an $\mathcal{O}$-orientation and the computation of $\mathcal{O}$-sesquilinear pairings for nice $m$  (Theorem~\ref{thm:equiv}).

\item Corollary~\ref{cor:festa} and Theorem~\ref{thm:sidh} (described in more detail below) provide evidence for a trade-off between the amount of known level structure of a secret isogeny $\phi : E \to E'$ of degree $d$ and how much of the endomorphism rings of $E$ and $E'$ we need to represent to find $\phi$.  As shown in \cite{Wesolowski}, the fixed-degree isogeny problem with full level structure is equivalent to finding a representation of the full endomorphism ring of $E$ and $E'$, while \cite{SIDHbreak1,SIDHbreak2,SIDHbreak3} show that the fixed-degree isogeny problem with minimal level structure requires no knowledge of even a partial representation of the endomorphism rings of $E$ and $E'$. As described in \emph{Cryptographic contributions} items~\ref{item:sidh} and~\ref{item:festa} below, knowledge of an intermediate level structure can be combined with a representation of only ``half'' of the endomorphism rings of $E$ and $E'$, to provide attacks on hidden isogenies of known degree.  See also the work in \cite{Level}, which explores varying amounts of level structure.

	\end{enumerate}

\textbf{Cryptographic contributions.} 
\begin{enumerate}
	\item We extend the applicability of the (sometimes polynomial) attacks from \cite{attacks} on the class group action problem (Section~\ref{sec:ramified}).  These attacks run for smooth $m$ dividing the discriminant.  We recover these attacks using the new pairings in a slightly different way, with the advantage that our pairing computations do not require going to a large field extension.  This partially addresses one of the open questions of \cite[Section 7]{attacks}.  Example~\ref{ex:wouter} gives an explicit situation in which the reach of polynomial attacks is strictly extended.
	\item \label{item:sidh} We demonstrate a pairing-based reduction from SIDH$_1$ to SIDH in the oriented situation for $E[m]$, where $m$ is smooth and coprime to the discriminant (Theorem~\ref{thm:sidh}), resulting in an attack when $m^2 > \deg \phi$.  This partially addresses the first and second open problems in \cite[Section 7]{attacks}.  Existing attacks on SIDH$_1$ (which apply without orientation information) require $m > \deg \phi$. 
	\item \label{item:festa} We reduce the hard problem underlying FESTA \cite{festa} to finding an orientation of the secret isogeny $\phi: E \rightarrow E'$ (i.e. an orientation of both curves and the isogeny between them) (Corollary~\ref{cor:festa}).  This follows from an attack on the Diagonal SIDH Problem (Theorem~\ref{thm:festa}).
 	\item We show how these pairings, using orientation information, easily reveal \emph{partial} information on the image of a torsion point $P$ of order $m$ for $m$ smooth (Theorem~\ref{thm:Nlambda}).  This results in an algorithm to break class-group-based schemes by running the SIDH attack on $\sqrt{\deg(\phi)}$ candidate torsion points as images under $\phi$ (Remark~\ref{rem:sqrt}). 
	\item Our results should be considered a cautionary tale for the design of decisional problems based on torsion point images, such as in \cite{SiGamal}, since the possible images of torsion points is restricted.  We discuss this in Remark~\ref{rem:decisional}.
	\item In the supersingular case, we demonstrate a method of finding the secret isogeny in the presence of two independent known orientations (which amounts to an explicit subring of the endomorphism ring of rank $4$), provided the secret isogeny is oriented for both orientations.  This is not a surprise, as this problem could be solved by the KLPT algorithm if the endomorphisms are obtained by walking the graph (see \cite{path}, and also \cite{KLPT,Wesolowski}), but it provides a new method via a simple reduction to the SIDH problem.  (Section~\ref{sec:twoorient}.)
\end{enumerate}

\section{Background}
\label{sec:back}

\subsection{Notations.} We study elliptic curves, typically denoted $E$, $E'$ etc., defined over finite fields, denoted by $\FF$ in general.  Denote an algebraic closure of $\FF$ by $\bar \FF$. The identity on $E$ is denoted $\infty$, and $\End(E)$ is the endomorphism ring over $\overline{\FF}$. We study imaginary quadratic fields, denoted $K$ in general, and orders in such fields, denoted by $\mathcal{O}$, $\mathcal{O}'$ etc. Greek letters typically denote elements of the orders. We denote the norm of an element $\lambda$ of a given order by $N(\lambda)$. When considering the action of an element $\alpha \in \mathcal{O}$ on a point $P$, we write $[\alpha]P$. The Greek letter $\phi$ always refers to an isogeny and $\widehat{\phi}$ always denotes the dual isogeny of $\phi$. For ease of notation, we write $\phi P$ instead of $\phi(P)$. Throughout the paper, we write $\mu_m$ for the copy of $\mu_m$ in a finite field.

\subsection{Orientations.} We study $\mathcal{O}$-\emph{oriented} elliptic curves over finite fields, which are curves together with the information of an embedding $\iota: \mathcal{O} \rightarrow \End(E)$.  This extends to an embedding of the same name, $\iota:  K \rightarrow \QQ \otimes_\ZZ \End(E)$, and the $\mathcal{O}$-orientation is called \emph{primitive} if $\iota(K) \cap \End(E) = \iota(\mathcal{O})$.  If the index $[\iota(K) \cap \End(E) : \iota(\mathcal{O}) ]$ is coprime to $n$, we say the orientation is $n$-\emph{primitive}.  Given an $\mathcal{O}$-orientation, there is a unique $\mathcal{O}' \supseteq \mathcal{O}$ for which $\iota$ becomes a $\mathcal{O}'$-primitive orientation, namely $\iota(\mathcal{O}') = \iota(K) \cap \End(E)$.  Given an elliptic curve $E$ with an $\mathcal{O}$-orientation, we define the \emph{relative conductor} of $\mathcal{O}$ to be the index $[\mathcal{O}': \mathcal{O}]$, for which the orientation is $\mathcal{O}'$-primitive.

If $\phi: E \rightarrow E'$ is an isogeny between two $\mathcal{O}$-oriented elliptic curves $(E, \iota)$ and $(E', \iota')$ is such that $\phi \circ \iota(\alpha) = \iota'(\alpha) \circ \phi$ for all $\alpha \in \mathcal{O}$, then we say that $\phi$ is an oriented isogeny.  Throughout the paper, we will generally fix a single $\mathcal{O}$-orientation for any curve, so we will often drop the $\iota$, writing simply $[\alpha]$ for $\iota(\alpha)$, writing $\mathcal{O} \subseteq \End(E)$, and characterizing oriented isogenies as those for which $\phi \circ [\alpha] = [\alpha] \circ \phi$.  This saves on notation.

\subsection{Cyclic self-pairings} \label{sec:CHM}

CSIDH, introduced in \cite{CSIDH}, relies for its security on the presumed hardness of the following instance of the \textit{vectorization problem}: given a (large) prime $p \equiv 3 \pmod{4}$ and two supersingular curves $E$ and $E'$ over $\FF_p$, find the ideal class $[\mathfrak{a}]$ of $ \mathcal{O} = \ZZ[\sqrt{-p}]$ such that $E' = [\mathfrak{a}]E$. More generally, the vectorization problem can be phrased as follows: given two supersingular elliptic curves $E, E'$ primitively oriented by an imaginary quadratic order $\mathcal{O}$ and known to be connected by the action of the ideal class group cl$(\mathcal{O})$ of $\mathcal{O}$, find $[\mathfrak{a}] \in \text{cl}(\mathcal{O})$ such that $E' = [\mathfrak{a}]E$.  This is also known as a class-group-action problem.

This paper builds on the previous work of \cite{attacks}. There, the authors ask whether the attack on SIDH \cite{SIDHbreak1,SIDHbreak2,SIDHbreak3} that renders the protocol insecure can be applied to solving the vectorization problem. In brief, they note that one can treat the SIDH attack as an oracle, which when given the degree $d$ of a secret $\FF_q$-rational isogeny $\phi$ between curves $E$ and $E'$ defined over $\FF_q$ and knowledge of its action on a basis of $E[m]$ for $m$ coprime to $d$ and $m^2 > 4d$, returns $\phi$. Assuming the degree $d$ of $\phi$ is known, the question of whether this oracle can answer the vectorization problem therefore reduces to the question of whether one can determine the action of $\phi$ on a basis of $E[m]$ for suitable $m$. 

The following example, reproduced directly from \cite{attacks}, is instructive. Assume the context of CSIDH, i.e., that the relevant order is $\ZZ[\sqrt{-p}]$. Choosing $m$ to be the power of a small prime $l$ coprime to $d$ that splits in $\QQ(\sqrt{-p})$, $E[m]$ has a basis $\{P, Q\}$ consisting of eigenvectors of the Frobenius endomorphism $\pi_p$. Since by assumption the curves $E, E'$ and the isogeny $\phi$ are all defined over $\FF_p$, $E'[m]$ has a basis $\{P', Q'\}$ consisting of eigenvectors of $\pi_p$ and $\phi P = r P'$, $\phi Q = s Q'$ for $r,s \in \left ( \ZZ/m\ZZ \right )^\times$. The bilinearity and compatiblity with isogenies of the $m$-Weil pairing imply that

\begin{equation}
	\label{eqn:weilattack}
	e_m(P', Q') = e_m(P, Q)^{rs d}
\end{equation}

\noindent With Miller's algorithm \cite{Miller}, computation of the $m$-Weil pairing is efficient. Since $m$ is a power of a small prime, also efficient is computation of discrete logarithms. Thus, given knowledge of $d$ it remains to determine one of $r$ or $s$. Yet properties of the $m$-Weil pairing imply that $e_m(P', P') = 1$; thus, one cannot compute $e_m(P', P') = e_m(P, P)^{r^2 d}$ to find $r$.

In \cite{attacks}, this obstacle is surmounted via the construction of cyclic self-pairings that are compatible with $\mathcal{O}$-oriented isogenies. A cyclic self-pairing is a function $f$ defined on a finite cyclic subgroup $C$ of an elliptic curve $E/\FF$ with the property that 

\[f(r P) = f(P)^{r^2} \hspace{.25cm} \text{for all } P \in C \text{ and } r \in \ZZ. \]

\noindent When $E$ and $E'$ are two curves over $\FF$ with orientations of an imaginary quadratic order $\mathcal{O}$ by $\iota, \iota'$, respectively, two self-pairings $f$ and $f'$ defined on finite subgroups $C$ of $E$ and $C'$ of $E'$ are compatible with $\mathcal{O}$-oriented isogenies $\phi : E \to E'$ when $\phi(C) \subset C'$ and $f'(\phi P) = f(P)^{\deg \phi}$. If $\phi$ is an $\mathcal{O}$-oriented isogeny from $E$ to $E'$ of degree $d$ coprime to an integer $m$ such that $E$ and $E'$ have non-trivial cyclic self-pairings $f$ and $f'$ compatible with $\mathcal{O}$-oriented isogenies on cyclic subgroups $\langle P \rangle$, $\langle P' \rangle$, then 

\[f'(P') = f(P)^{d r^2}\]

\noindent for some $r \in \ZZ/m\ZZ$. If furthermore discrete logarithms are efficiently computable modulo $m$, then the non-triviality of $f$ and $f'$ implies that one can efficiently determine $r^2$ modulo $m$. Assuming $m$ has a nice factorization, this leaves only a few possibilities for $r$, and one simply guesses by direct computation which one is correct.

The non-trivial self-pairings constructed in \cite{attacks} are generalizations on the classical reduced $m$-Tate pairing. We refer the reader to Section 5 of \cite{attacks} for further details. Crucially, the order $m$ of non-trivial cyclic self-pairings compatible with $\mathcal{O}$-oriented isogenies must divide $\Delta_{\mathcal{O}}$ (\cite{attacks}, Proposition 4.8). Furthermore, the existence of such a self-pairing only yields knowledge of the image of a single torsion point $P$ under $\phi$. In \cite{attacks}, this latter issue is addressed by assuming $m^2 \mid \Delta_{\mathcal{O}}$. Then one obtains the image of an order $m^2$ point $P$ under $\phi$. The authors briefly describe how one can then obtain from this data knowledge of the action of $\phi$ on a basis for $E[m]$. A more systematic description of this reduction in the language of level structure is in Section 5 of \cite{Level} (in particular, see corollary 12).

Thus, there are two significant limitations to the scope of this attack. First, $\Delta_{\mathcal{O}}$ must contain a large smooth square factor. Second, in general one must work over a field where the $m^2$-torsion is fully rational. In the worst case, this requires a base change to an extension of potentially large degree.

The first of these limitations is addressed in work in preparation by Castryck et. al. \cite{unpubCastryck}, the authors show that with knowledge of the image of $\phi : E \to E'$ on a generating set $S$ for a subgroup $G$ with $\#G > 4d$, there is an algorithm to determine $\phi$ (in the sense of computing arbitrary images) that is polynomial in \begin{enumerate} \item[(1)] the size of the $S$ and $\log q$, where $q$ is the size of the field over which $E$ and $E'$ are defined; \item[(2)] the size of the largest prime factor of $\#G$;\item[(3)] the largest degree of the fields of definition of $E[\ell^{\lfloor e/2 \rfloor }]$, taken over all prime powers $l^e$ dividing $\#G$.
\end{enumerate} In particular, this result obviates the need for $\Delta_{\mathcal{O}}$ to contain a smooth square factor; instead, one only needs a smooth factor of size greater than $4d$.

\subsection{Level Structure} Many isogeny-based protocols require that some torsion-point information be made publicly known. For example, in SIDH, the image under the secret isogeny $\phi$ of a specified basis $\{P, Q\}$ for $E[m]$ (where $m$ is a power of a small prime coprime to the characteristic of the field $\FF$ over which $E$ is defined) is known. As discussed in the last section, in CSIDH the image of a basis $\{P, Q\}$ for $E[m]$ (with $m$ again a power of a small prime coprime to the field characteristic, but also coprime to the degree $d$ of the secret isogeny $\phi$) is known, up to multiplication by an element of $\left ( \ZZ/m\ZZ \right )^\times$. Equivalently, the image under $\phi$ of two order $m$ subgroups of $E$ is known. Both types of torsion-point information are examples of \textit{level structure} that $\phi$ respects.

\begin{definition} [\cite{Level}]
    Let $E$ be an elliptic curve over a finite field $\FF$ of characteristic $p$ and $m$ be a positive integer coprime to $p$. Let $\Gamma$ be a subgroup of $\GL_2(\ZZ/m\ZZ)$. A \emph{$\Gamma$-level structure of level $m$ on $E$} is a $\Gamma$-orbit of a basis of $E[m]$.
\end{definition}

Typically in the context of isogeny problems, one is not interested in level structure \emph{per se}, but in level structure that a given isogeny $\phi$ respects. That is, given curves $E$ and $E'$ both with $\Gamma$-level $m$ structures for a fixed $\Gamma$, $\phi$ maps the specified $\Gamma$-orbit for $E[m]$ to the specified $\Gamma$-orbit for $E'[m]$. 

There has been much attention paid recently to elliptic curves equipped with a particular level structure. Arpin \cite{Arp24} studies the correspondence of Eichler orders in the quaternion algebra $B_{p, \infty}$ with supersingular elliptic curves over $\overline{\FF}_p$ equipped with \textit{Borel} level structure---i.e., where $\Gamma = \left ( \{\begin{smallmatrix} * & 0 \\ * & *
\end{smallmatrix} \} \right )$---for $m$ squarefree and coprime to $p$. In \cite{CL24}, the authors consider the structure of the supersingular isogeny graph with varying level structures and show that many of these graphs remain Ramanujan.  Others investigate the actions of generalized ideal class groups on elliptic curves over finite fields with level structure \cite{GPV23,ACE+24}. In this paper, we are primarily interested in level structure as a unifying framework for understanding the security of various proposals in isogeny-based cryptography. This framework is described in detail in \cite{Level}. In particular, those authors make explicit the implicit level structures in several schemes including SIDH, M-SIDH, CSIDH, and FESTA, and prove several security reductions between various level structures.

\subsection{Computational assumptions} With regards to computations, we use the word \emph{efficient} to mean polynomial time in the size of the input, which is itself typically a polynomial in $\log m$ (the torsion) and $\log q$ (where $q$ is the cardinality of the field of definition of the $m$-torsion), in our context. Throughout the paper, when we assume that we are given an $\mathcal{O}$-oriented elliptic curve, we mean that we are given an explicit orientation, and in particular that, given an element $\alpha \in \mathcal{O}$, we can compute its action $[\alpha]$ on a point $P$ on $E$ efficiently.  

We assume throughout that the degree of the hidden isogeny is known.

We assume that $m$ is coprime to the characteristic $p$ of the given field $\FF$, and that $m$ is smooth, meaning that its factors are polynomial in size, so that discrete logarithms in $\mu_m$ or $E[m]$ are computable in polynomial time.  In particular, we can efficiently write any element of $E[m]$ in terms of a given basis.

\subsection{The Tate-Lichtenbaum Frey-R\"uck Pairing.}  We review the definition and basic properties of the Tate-Lichtenbaum pairing.

\begin{definition}
\label{defn: tate}
Let $m > 1$ be an integer.  Let $E$ be an elliptic curve defined over a field $\FF$ (assumed finite in this paper).  Suppose that $P \in E(\FF)[m]$.  Choose divisors $D_P$ and $D_Q$ of disjoint support such that $D_P \sim (P) - (\mathcal{O})$ and $D_Q \sim (Q) - (\mathcal{O})$.
Then $mD_P \sim \emptyset$, hence there is a function $f_P$ such that $\operatorname{div}(f_P) = mD_P$.
The Tate-Lichtenbaum pairing
\[
t_m: E(\FF)[m] \times E(\FF)/mE(\FF) \rightarrow \FF^* / (\FF^*)^m
\]
is defined by
\[
t_m(P,Q) = f_P(D_Q).
\]
\end{definition}

The standard properties of the Tate pairing are as follows.  Proofs can be found in many places, for example \cite{Robert} and \cite[Sec 3.2]{attacks}.

\begin{proposition}
	\label{prop:tate}
Definition \ref{defn: tate} is well-defined, and has the following properties:
\begin{enumerate}
\item Bilinearity: for $P,P' \in E(\FF)[m]$ and $Q,Q' \in E(\FF)$
\begin{align*}
t_m(P+P',Q) =& t_m(P,Q)t_m(P',Q), \\
t_m(P,Q+Q') =& t_m(P,Q)t_m(P,Q').
\end{align*}
\item Non-degeneracy: Let $\FF$ be a finite field containing the $m$-th roots of unity $\mu_m$.  For nonzero $P \in E(\FF)[m]$, there exists $Q \in E(\FF)$ such that
\[
t_m(P,Q) \neq 1.
\]
Furthermore, for $Q \in E(\FF) \backslash mE(\FF)$, there exists a $P \in E(\FF)[m]$ such that
\[
t_m(P,Q) \neq 1.
\]
In particular, for $P$ of order $m$, there exists $Q$ such that $t_m(P,Q)$ has order $m$, and similarly for the other entry.
\item Compatibility:  For a point $P \in E(\FF)[m]$, an isogeny $\phi:E \rightarrow E'$, and a point $Q \in E'(\FF)$,
\[
t_m(\phi P, Q) = t_m(P, \widehat{\phi} Q).
\]
\end{enumerate}
\end{proposition}

\section{Structure of $E[\alpha]$}
\label{sec:structure}

Suppose $E$ has an $\mathcal{O}$-orientation. Let $\alpha \in \mathcal{O}$. We wish to know when $E[\alpha]$ is cyclic as an $\mathcal{O}$-module. The following two theorems of Lenstra are relevant.

\begin{theorem}[{\cite[Proposition 2.1]{Lenstra}}]
	\label{thm:lenstra1}
Let $E$ be an elliptic curve over an algebraically closed field $k$, and $\mathcal{O}$ a subring of $\End_k(E)$ such that as $\ZZ$-modules, $\mathcal{O}$ is free of rank 2 and $\End_k(E)/\mathcal{O}$ is torsion-free. Then for every separable element $\alpha \in \mathcal{O}$, $E[\alpha] \cong \mathcal{O}/\alpha \mathcal{O}$ as $\mathcal{O}$-modules. 
\end{theorem}

When $\alpha$ is inseparable, Lenstra has a similar result. With $\mathcal{O}$ as above, $\text{char } k = p > 0$, and $K = \mathcal{O} \otimes_\ZZ \QQ$, he observes that there is a $p$-adic valuation $\nu$ on $K$ with $\nu(\alpha) = \log(\deg_i \alpha)/\log p$ for $\alpha \in \mathcal{O}$. Following his notation, we define $V = \{x \in K : \nu(x) \geq 0\}$. 

\begin{theorem}[{\cite[Proposition 2.4]{Lenstra}}] 
	\label{thm:lenstra2}
	Let the notation and hypotheses be as above. Then for every non-zero element $\alpha \in \mathcal{O}$ there is an isomorphism of $\mathcal{O}$-modules $E[\alpha] \oplus (V/\alpha V) \cong \mathcal{O}/\alpha \mathcal{O}$.
\end{theorem}

\begin{theorem}
	\label{thm:cyclic}
	Let $E$ be an elliptic curve over $\FF$, $K$ an imaginary quadratic field, and $\mathcal{O}$ an order in $K$ such that $E$ has an $\mathcal{O}$-orientation, which is primitive when extended to $\mathcal{O'}$.  Let $f = [\mathcal{O}':\mathcal{O}]$ be the relative conductor of $\mathcal{O}$. For $\alpha \in \mathcal{O}$ with $N(\alpha)$ coprime to $f$ (i.e., such that the $\mathcal{O}$-orientation is $N(\alpha)$-primitive), then $E[\alpha]$ is cyclic as an $\mathcal{O}$-module. Specifically:

    \begin{enumerate}
        \item If $\alpha$ is separable, then $E[\alpha] \cong \mathcal{O}/\alpha \mathcal{O}$.

        \item If $\alpha$ is inseparable, then $E[\alpha]$ is isomorphic to a proper cyclic $\mathcal{O}$-submodule of $\mathcal{O}/\alpha \mathcal{O}$.
    \end{enumerate}

    As a partial converse, as soon as $\alpha$ factors through multiplication by $n$ for some $n > 1$ that divides $f$, $E[\alpha]$ is not cyclic as an $\mathcal{O}$-module. In particular, if $\alpha = m \in \ZZ$, then $E[m]$ is a cyclic $\mathcal{O}$-module if and only if $m$ and $f$ are coprime. 
    
\end{theorem}

\begin{proof}
	Suppose first that $\alpha$ is separable. Let $\iota$ be an $\mathcal{O}$-orientation for $E$ and $\mathcal{O'}$ be the order of $K$ for which $\iota$ is a primitive orientation. Then as an abelian group the quotient $\End(E)/\mathcal{O'}$ is torsion-free and Theorem~\ref{thm:lenstra1} tells us that $E[\alpha] \cong \mathcal{O'}/\alpha \mathcal{O'}$ as $\mathcal{O}'$-modules. We have $N(\alpha) = N(\alpha \mathcal{O'}) = \left | \mathcal{O'}/\alpha \mathcal{O'} \right | $, so since $N(\alpha)$ is coprime to $f$, it follows from \cite[Proposition 7.18, 7.20]{cox22} that the natural injection $\mathcal{O}/\alpha \mathcal{O} \rightarrow \mathcal{O'}/\alpha \mathcal{O'}$ is an isomorphism of $\mathcal{O}$-modules.

	Suppose then that $\alpha$ is inseparable. Let $\mathcal{O'}$ be as above. From Theorem~\ref{thm:lenstra2} we have $E[\alpha] \oplus V/\alpha V \cong \mathcal{O'}/\alpha \mathcal{O'}$, so $E[\alpha]$ is isomorphic as an $\mathcal{O'}$-module to $(\mathcal{O'}/\alpha \mathcal{O'})/(V/\alpha V)$ and hence is a cyclic $\mathcal{O'}$-module. Since  $\mathcal{O'}/\alpha \mathcal{O'} \cong \mathcal{O}/\alpha \mathcal{O}$ as $\mathcal{O}$-modules, again by our assumption that $N(\alpha)$ is coprime to $f$, it follows that $E[\alpha]$ is cyclic as an $\mathcal{O}$-module. 
    
	Finally, suppose $\alpha$ factors through $[n]$ for some $n > 1$ with $n \mid f$. Then as a $\ZZ$-module, $E[\alpha] \cong \ZZ/b\ZZ \times \ZZ/c\ZZ$ with $n \mid b \mid c$. Let $Q$ be an arbitrary point of order $c$ in $E[\alpha]$ and extend to a generating set $\{P, Q\}$ for $E[\alpha]$ with $\operatorname{ord}(P) = b, \operatorname{ord}(Q) = c$. Let $\mathcal{O'} = \ZZ[\sigma]$ for some $\sigma$.  Then $\mathcal{O} = \ZZ[f\sigma]$. Since any element of $\mathcal{O}$ is a $\ZZ$-linear combination of $[1]$ and $[f\sigma]$, whether or not $E[\alpha]$ is cyclic as an $\mathcal{O}$-module is determined by the action of $f \sigma$. We have 
$[f\sigma]Q 
 = [n\sigma]Q' $,
 where $Q' = [f/n]Q$. 
 
 If $[\sigma]Q' = [s]P + [t]Q$, then $[f\sigma] Q = [ns]P + [nt]Q$, and we cannot obtain $P$ from the action of any $\ZZ$-linear combination of $[1]$ and $[f\sigma]$ on $Q$ (since $[n]$ is not injective on $E[\alpha]$). Thus, $Q$ cannot be a generator for $E[\alpha]$ as an $\mathcal{O}$-module. Since $Q$ was an arbitrary order $c$ point, and since no point of order strictly less than $c$ can generate $E[\alpha]$ as an $\mathcal{O}$-module (endomorphisms send points of $E$ of order $m$ to points of $E$ of order dividing $m$), $E[\alpha]$ cannot be a cyclic $\mathcal{O}$-module. \qed
\end{proof}

\begin{example}
	Consider the ordinary curve $y^2 = x^3 + 30x + 2$ over $\FF_{101}$. Denoting the Frobenius endomorphism of degree $p$ by $\pi$, $\ZZ[\pi]$ has conductor 2 in the maximal order and $[\ZZ[\pi] : \ZZ[\pi^2]] =18$.  Thus, Theorem~\ref{thm:cyclic} implies $E[3]$ is not cyclic as a $\ZZ[\pi^2]$-module. Indeed, making a base change to $\FF_{101^2}$, the 3-torsion of $E$ becomes rational. On the other hand, $E[3]$ is cyclic as a $\ZZ[\pi]$-module. With $\FF_{101^2} = \FF_{101}(a)$ and $x^2 -4x + 2$ the minimal polynomial of $a$, we have $P = (41a + 16, 39a + 19) \in E[3]$ and $\pi(P) = (60a + 79, 62a + 74) \not\in \langle P \rangle$, hence $\ZZ[\pi] P = E[3]$.
\end{example}

\section{Sesquilinear pairings}
\label{sec:sesqui}

We follow \cite{katepairing} in this section.
Suppose $\mathcal{O} = \ZZ[\tau]$ is an imaginary quadratic order.
Let $E$ have CM by $\mathcal{O}$.  Let
$\rho: \mathcal{O} \rightarrow M_{2 \times 2}(\ZZ)$
be the left-regular representation of $\mathcal{O}$ acting on the basis $1$ and $\tau$, i.e. 
\[
\rho(\alpha) = \begin{pmatrix} a & b \\ c & d \end{pmatrix} \quad \iff \quad \alpha = a + c\tau, \quad \alpha \tau = b + d\tau.
\]
Then we endow the Cartesian square $(\FF^*)^{\times 2}$ of the multiplicative $\ZZ$-module $\FF^*$ (i.e. $\ZZ$-coefficients in the exponent) with a multiplicative $\mathcal{O}$-module action (i.e. $\mathcal{O}$-coefficients in the exponent) via 
\begin{equation}
    \label{eqn:action}
(x,y)^\alpha = \rho(\alpha) \cdot (x,y) = (x^ay^b,x^cy^d), \quad \text{where} \quad \rho(\alpha) = \begin{pmatrix} a & b \\ c & d \end{pmatrix}.
\end{equation}

In the case of an $\mathcal{O}$-module, by \emph{order} of an element we mean the $\ZZ$-order; we can also discuss the annihilator as an $\mathcal{O}$-module, which may be distinct from this.

For each $\alpha \in \mathcal{O}$, we define a bilinear pairing
\[
	\widehat{T}^{\tau}_\alpha : E[\overline\alpha](\FF) \times E(\FF) / [\alpha]E(\FF) \rightarrow (\FF^*)^{\times 2} / ((\FF^*)^{\times 2})^{\alpha}
\]
as follows.  Write
\[
\rho({\alpha}) = \begin{pmatrix} a & b \\ c & d \end{pmatrix}, \quad
\rho(\overline{\alpha}) = \begin{pmatrix} d & -b \\ -c & a \end{pmatrix}.
\]
Observe that this corresponds to the ring facts
\[
a + c\tau = \alpha, \quad b + d\tau = \alpha \tau, \quad
d - c\tau = \overline{\alpha}, \quad -b + a\tau = \overline{\alpha}\tau.
\]
We take $P \in E[\overline{\alpha}]$, 
Define $f_P = (f_{P,1}, f_{P,2})$, where
\begin{align*}
    \label{eqn:hpgp-2}
    \div(f_{P,1}) &= a([-\tau]P) + b(P)-(a+b)(\infty), \\
    \div(f_{P,2}) &= c
([-\tau]P) + d(P) - (c+d)(\infty).
\end{align*}
Choose an auxiliary point $R \in E(\overline\FF)$ and define for $Q \in E(\FF)$,
\[
D_{Q,1} = ([-\tau]Q + [-\tau]R) - ([-\tau]R), \quad
D_{Q,2} = (Q + R) - (R).
\]
Then, choosing $R$ so that the necessary supports are disjoint (i.e. the support of $\div(f_{P,i})$ and $D_{Q,j}$ are disjoint for each pair $i$, $j$), the pairing is defined (using \eqref{eqn:action}) as
\[
\widehat{T}^{\tau}_\alpha(P,Q) := \left(
f_{P,1}(D_{Q,1}), f_{P,2}(D_{Q,1})
\right)
\left(
f_{P,1}(D_{Q,2}), f_{P,2}(D_{Q,2})
\right)^{\overline\tau}
\]
which can also be expressed as
\[
\left(
f_{P,1}(D_{Q,1})f_{P,1}(D_{Q,2})^{Tr(\tau)}f_{P,2}(D_{Q,2})^{N(\tau)}, f_{P,2}(D_{Q,1})f_{P,1}(D_{Q,2})^{-1}
\right).
\]

\begin{remark}
In \cite{katepairing}, 
it is shown how it is possible to think of these definitions as elements of $\mathcal{O} \otimes_\ZZ \Pic^0(E)$:
\[
D_Q = D_{Q,1} + \tau \cdot D_{Q,2}, \quad D_P = ([-\tau]P)-(\infty) + \tau \cdot ( (P) - (\infty) ),
\]
and analogously define $f_P$ satisfying $\div(f_P) = {\alpha} \cdot D_P$, so that the definition above has the form $f_P(D_Q)$ as for the classical Tate pairing, and the apparent dependence on the basis $1$, $\tau$ disappears.  For simplicity here, we stick to the direct definition above.  In that same paper, analogous constructions are also given for quaternion orders and Weil-like pairings.
\end{remark}

\begin{theorem}[{\cite[Theorems 5.4, 5.5, 5.6]{katepairing}}]
\label{thm:tate-sesqui}
The pairing above is well-defined and satisfies
    \begin{enumerate}
        \item Sesquilinearity:  For $P \in E[\overline\alpha](\FF)$ and $Q \in E(\FF)$,    \[
    \widehat{T}^{\tau}_\alpha([\gamma] P, [\delta] Q)
    = \widehat{T}^{\tau}_\alpha(P,Q)^{{\overline\gamma}{\delta}}.
    \]
        \item Compatibility:   
            Let $\phi: E \rightarrow E'$ be an isogeny between curves with CM by $\mathcal{O}$ and satisfying $[\alpha] \circ \phi = \phi \circ [\alpha]$.   Then for $P \in E[\overline\alpha](\FF)$ and $Q \in E(\FF)$,
    \[
    \widehat{T}^{\tau}_\alpha(\phi P,\phi Q) = \widehat{T}^{\tau}_\alpha(P,Q)^{\deg \phi}.
    \]
    \item Coherence:
	    Suppose $P \in E[\overline{\alpha \beta}](\FF)$, and $Q \in E(\FF)/ [\alpha\beta] E(\FF)$.  Then
    \[
    \widehat{T}^{\tau}_{\alpha \beta}(P, Q) \bmod{ ((\FF^*)^{\times 2})^{\beta}}
    = \widehat{T}^{\tau}_{\beta}( [\overline\alpha]P, Q \bmod [{\beta}]E ).
    \]
    Suppose $P \in E[\overline{\alpha}](\FF)$, and $Q \in E(\FF)/ [\alpha\beta] E(\FF)$.  Then
    \[
    \widehat{T}^{\tau}_{\alpha \beta}(P, Q) \bmod{ ((\FF^*)^{\times 2})^{\alpha}}
    = \widehat{T}^{\tau}_{\alpha}( P, [{\beta}]Q \bmod [{\alpha}]E ).
    \]
        \item \label{item:nondeg} Non-degeneracy: 
    Let $\FF$ be a finite field, and let $E$ be an elliptic curve defined over $\FF$.  Let $\alpha \in \mathcal{O}$ be coprime to $char(\FF)$ and the discriminant of $\mathcal{O}$.  Let $N = N(\alpha)$.  Suppose $\FF$ contains the $N$-th roots of unity.  
     Suppose there exists $P \in E[N](\FF)$ such that $\mathcal{O} P = E[N] = E[N](\FF)$.
    Then 
    \[
\widehat{T}^{\tau}_\alpha : E[\overline{\alpha}](\FF) \times E(\FF) / [\alpha] E(\FF) \rightarrow (\FF^*)^{\times 2} /  ((\FF^*)^{\times 2})^{\alpha},
\] is non-degenerate.    Furthermore, if $P$ has annihilator $\overline\alpha \mathcal{O}$, then $T_\alpha(P, \cdot)$ is surjective; and if $Q$ has annihilator $ \alpha \mathcal{O}$, then  $T_\alpha(\cdot, Q)$ is surjective.
\item \label{item:tn}
Let $t_n$ be the $n$-Tate-Lichtenbaum pairing as described in Section~\ref{sec:back}.  Then
    \[
     \widehat{T}^{\tau}_n (P,Q) = 
   \left( 
	   t_n(P,Q)^{2N(\tau)}
	   t_n([-\tau]P,Q)^{Tr(\tau)}, \;
	   t_n([\tau - \overline{\tau}]P,Q)
    \right).
    \]
    \item \label{item:talpha}  Provided both of the following quantities are defined,
    \[
    \widehat{T}^{\tau}_{N(\alpha)}(P,Q) = \widehat{T}^{\tau}_\alpha(P,Q)^{\overline\alpha} \pmod{((\FF^*)^{\times 2})^{\alpha}}
    \]
\end{enumerate}
\end{theorem}

\begin{theorem}
\label{thm:compute-pair}
	Provided the action of $\tau$ is efficiently computable, then the pairing $\widehat{T}_\alpha^\tau(P,Q)$ is efficiently computable.  That is, it takes polynomially many operations in the field of definition of $P$ and $Q$.
\end{theorem}

\begin{proof}
	This follows from the definition given above, which is amenable to a Miller-style pairing algorithm; details are in \cite[Algorithm 5.7]{katepairing}.\qed
\end{proof}

To use the pairings $\widehat{T}^{\tau}_\alpha$, the most expedient computation method is the formulas given in Theorem~\ref{thm:tate-sesqui} items \eqref{item:tn} and \eqref{item:talpha}.  In particular, in our applications of $\widehat{T}^{\tau}_\alpha$ to form a discrete logarithm problem, in most use cases it suffices to compute $\widehat{T}^{\tau}_\alpha(P,Q)^{\overline\alpha}$ instead.  But if one wishes, one can compute $\overline\alpha^{-1} \pmod{\alpha}$ (provided $\alpha$ and $\overline\alpha$ are coprime), and use
    \[
	    \widehat{T}^{\tau}_\alpha(P,Q) = \widehat{T}^{\tau}_{N(\alpha)}(P,Q)^{\overline\alpha^{-1}} \pmod{((\FF^*)^{\times 2})^{\alpha}}.
    \]
    This may not apply when $\alpha$ divides the discriminant.

    \begin{definition}
	    \label{def:red}
    Also for cryptographic applications, it is convenient to apply a final exponentiation to obtain a \emph{reduced pairing}, as is common with the classical Tate pairing.  This will move the pairing into the roots of unity:
    \[
	    (\overline \FF^*)/(\overline \FF^*)^\alpha \rightarrow \mu_{N(\alpha)}^{\times 2} \subseteq (\overline \FF^*)^{\times 2}, \quad
	    x \mapsto x^{(q-1)\alpha^{-1}}.
    \] 
    \end{definition}

\begin{lemma}
	\label{lem:squares}
	Consider the image $N(\mathcal{O})$ of $\mathcal{O}$ under the norm map.  Then $N(\mathcal{O})$ modulo $m > 2$ is a subset of $\{ x^2 : x \in \ZZ/m\ZZ \}$ only if $m$ and $\Delta_\mathcal{O}$ share a non-trivial factor. 
\end{lemma}

\begin{proof}
	We may assume, by Sunzi's Theorem (Chinese Remainder Theorem), that $m$ is a prime power $p^k > 2$.
	If $p$ is split, the statement follows from the fact that the norm map from $\QQ_p\otimes_\ZZ\mathcal{O}$ to $\QQ_p$ is surjective.  If $p$ is inert, the norm map 
 is surjective on the residue field, so $N(\mathcal{O})$ modulo $p^k$ does include non-squares. 
\qed
\end{proof}

 \begin{theorem}
	    \label{thm:order}
	    Let $E$ be an elliptic curve oriented by $\mathcal{O} = \ZZ[\tau]$.  Let $m$ be coprime to the discriminant $\Delta_{\mathcal{O}}$.
	    Let $\FF$ be a finite field containing the $m$-th roots of unity.  Suppose $E[m] = E[m](\FF)$.  
	    Let $P$ have order $m$.
	    Let $s$ be the maximal divisor of $m$ such that $E[s] \subseteq \mathcal{O}P$.
	    Then the multiplicative order $m'$ of $\widehat{T}_m^\tau(P,P)$ satisfies $s \mid m' \mid 2s^2$.
	    In particular, if $\mathcal{O}P = E[m]$, then $s=m$ and the self-pairing has order $m$.  If $\mathcal{O}P = \ZZ P$, then $s=1$, and in fact, in this case, the self-pairing is trivial.
    \end{theorem}

    To rephrase the last sentence, the self-pairing is trivial on the eigenspaces for the action of $\mathcal{O}$ on $E[m]$.  This observation by itself is a consequence of the classification of self-pairings in \cite{attacks}.

    \begin{proof}
	    Let $m' \mid m$ be the order of $\widehat{T}_m^\tau(P,P)$.  Suppose $s$ is the maximal divisor of $m$ so that $E[s] \subseteq \mathcal{O}P$. 
	    In other words, $\mathcal{O}P \cong \ZZ/s\ZZ \times \ZZ/m\ZZ$ and $\mathcal{O}P / \ZZ P \cong \ZZ/s\ZZ$ as abelian groups.
	    In particular, $[s]\mathcal{O} P \in \ZZ P$. Thus $\mathcal{O} [s]P = \ZZ [s]P$.  
		
        We will show that $m' \mid 2s^2$ and $s \mid m'$. Let $\lambda \in \mathcal{O}$.  Then $[\lambda s]P = [ks]P$ for some $k = k(\lambda) \in \ZZ$, and then
	    \[
		    \widehat{T}_m^\tau([s]P,[s]P)^{k^2} =
		    \widehat{T}_m^\tau([ks]P,[ks]P) =
	    \widehat{T}_m^\tau([\lambda s]P,[\lambda s]P) =
	    \widehat{T}_m^\tau([s]P,[s]P)^{N(\lambda)}.
	    \]
	    Ranging over all $\lambda \in \mathcal{O}$, we conclude that $N(\lambda)$ are squares modulo $m'' := m'/\gcd(m',s^2)$, the multiplicative order of $\widehat{T}_m^{\tau}([s]P,[s]P)$, contradicting that $m$ is coprime to the discriminant
	    unless $m''=1$ or $2$ by Lemma~\ref{lem:squares}.  Therefore $m' \mid 2s^2$. In the case where $s = 1$, this argument implies only that the order of $\widehat{T}^\tau_m(P, P)$ is at most 2. However, the fact that in this case the order of the self-pairing is trivial follows immediately from \cite[Proposition 4.8]{attacks}.
	    
	    On the other hand, by
Theorem~\ref{thm:tate-sesqui} item \eqref{item:nondeg}, there exists some $Q$ so that $\widehat{T}_m^\tau(P,Q)$ has order $m$. Let $t = m/s$.  Then there is a basis for $E[m]$ of the form $P, P'$ where $[t]P' = [\lambda]P$ for some $\lambda \in \mathcal{O}$.  Writing $Q = [a]P + [b]P'$,
	    \[
		    \widehat{T}_m^\tau(P,Q)^{t} 
		    = \widehat{T}_m^\tau(P,[t]([a]P+[b]P'))
		    = \widehat{T}_m^\tau(P,[ta + b\lambda]P)
		    = \widehat{T}_m^{\tau}(P,P)^{ta+b\lambda}.
	    \]
	    This has order $s$ on the left.  Therefore $\widehat{T}_m^\tau(P,P)$ must have order a multiple of $s$.  Hence $s \mid m'$.
\qed
    \end{proof}

    \begin{remark}
        As discussed in Section~\ref{sec:CHM}, the authors of \cite{attacks}, the authors show that non-trivial cyclic self-pairings can only exist for $P$ of order $m$ dividing $\Delta_\mathcal{O}$.  The reason our pairings are not ruled out by this result is that our pairings are defined not only on \emph{cyclic subgroups stabilized by the orientation} (where they are in fact trivial, as required). 
    \end{remark}

    The following is a partial converse to Theorem~\ref{thm:compute-pair}.

    \begin{theorem}
	    \label{thm:equiv}
	    Let $E$ be an elliptic curve defined over a finite field $\FF$, and let $m \in \ZZ$.  Let a basis for $E[m]$ be given.
	    Suppose arithmetic in $\FF$, discrete logarithms in $\FF^*$ modulo $m$, and group law computations on $E[m]$ can all be accomplished in polynomial time.  
     Suppose $\varphi(m) > \sqrt{2/3}m$. 
	    Suppose $E$ is known to be oriented by $\mathcal{O} = \ZZ[\tau]$ (but the orientation $\iota$ is not given), and suppose $m$ is coprime to the discriminant $\Delta_\mathcal{O}$. Then the computation of arbitrary pairings $\widehat{T}^{\tau}_m(P,Q)$ on $E[m]$ is Monte-Carlo equivalent in polynomial time to the computation of the action of $[\tau]$ on $E[m]$.
    \end{theorem}

    By Monte-Carlo equivalent, we mean that there is an arbitrarily small probability that the algorithm will return incorrectly.  The condition on $\varphi(m)$ can be improved:  what should be required is that $\varphi(m)$ non-negligibly exceed $m/\sqrt{2}$.

    \begin{proof}
	    Note that computation of $[\tau]$ allows for computation of $[\overline{\tau}] = [Tr(\tau)] - [\tau]$. 

	    If $[\tau]$ is computable, then by Theorem~\ref{thm:tate-sesqui}~\eqref{item:tn} one can compute $\widehat{T}^{\tau}_m(P,Q)$ by computing $3$ multiplications by $[\tau]$ or $[\overline\tau]$, one addition, and $3$ classical Tate pairings.

	    Conversely, suppose one can compute $\widehat{T}^{\tau}_m(P,Q)$ for any $P, Q \in E[m]$.  We will show how to compute the action of $[\tau]$ on $E[m]$.  The pairing is non-degenerate as a consequence of the given hypotheses. It is possible to sample randomly from the subset of order $m$ points in $E[m]$, by choosing $P$ uniformly randomly as a linear combination $aP_1 + bP_2$ of the given basis $P_1, P_2$ such that $\gcd(a,b)$ is coprime to $m$. Choose such a $P$ of order $m$ and compute $\widehat{T}^{\tau}_m(P,P)$.

	    For now, we assume that $\mathcal{O}P = E[m]$. Choose $Q \in E[m]$ so that $P, Q$ form a basis for $E[m]$.
	    Then $Q = [\lambda] P$ for some $\lambda \notin \ZZ$; then
	    \[
		    \widehat{T}^{\tau}_m(P,Q) = \widehat{T}^{\tau}_m(P,P)^\lambda.
		    \]
		    Since $\widehat{T}^{\tau}_m(P,P)$ is of order $m$ by Theorem~\ref{thm:order}, we can compute $\lambda$ modulo $m$ by two pairing computations and a discrete logarithm in $(\FF^*)^{\times 2}/((\FF^*)^{\times 2})^m \cong \mu_m$. 

	By construction, we can write $\tau = a + b \lambda$ modulo $m$, so we can compute $[\tau]P = [a]P + [b]Q$.

	To compute $[\tau]R$ for arbitrary $R$, we first determine $\mu \in \mathcal{O}$ modulo $m$ such that $R = [\mu] P$ (we may use the same discrete log method as above), and then we have $[\tau]R = [\mu][\tau]P$.

	If $\mathcal{O}P \neq E[m]$, then the algorithm is not guaranteed to be correct.  Therefore, we run the algorithm several times using different random $P$ of order $m$.  We have $E[m] \cong \mathcal{O}/m\mathcal{O}$ by Theorem~\ref{thm:cyclic}.  Any element of $\mathcal{O}$ is a $\mathcal{O}$-module generator of $\mathcal{O}/m\mathcal{O}$ provided it is coprime to $m$ (since $1$ is a generator and it has an inverse modulo $m$).  So the proportion of such generators is at least $(\varphi(m)/m)^2$.  By our assumption on $m$, this exceeds $2/3$. Any such $P$ has self-pairing of order $m$ (by non-degeneracy), so repeating sufficiently often and taking the majority rule answer, this will succeed with overwhelming probability in polynomial time. 
\qed
	\end{proof}

\begin{remark}
    If $\varphi(m)/m$ is non-negligible, then one can sample points uniformly at random and use the pairing to check whether they generate $\mathcal{O}$, with a high probability of success.  However, if $m$ is badly behaved, for example, a primorial, then $\varphi(m)/m$ may be less than $1/m^x$ for some $x > 0$.
\end{remark}

	\begin{remark}
		\label{rem:eigen}
		Given any basis for $E[m]$, the pairing $\widehat{T}_m^\tau$ allows us to compute the `eigenspaces', i.e. a basis $P, Q$ such that $[\tau]P \in \ZZ P$ and $[\tau]Q \in \ZZ Q$.  That is, knowing the pairing values on the original basis, we can solve for points with trivial self-pairing.
	\end{remark}

	\begin{example}
		Consider the elliptic curve $y^2 = x^3 + x$ over $\FF_p$, $p=541$.  A basis for $E[5]$ is $P = (109, 208)$, $Q=(53, 195)$.  If we compute the self-pairings $\widehat{T}_5^{[i]}([a]P + [b]Q,[a]P+[b]Q)$, for $a,b = 0, \ldots, 4$, we obtain the following:  the left matrix shows the real parts and the right matrix the imaginary parts, taken to the $\log$ base 48 (48 is a generator of $\FF_{541}^*$).  So, for example, the fourth entry ($a=3$) in the second column ($b=1$) (of both matrices) indicates that $\widehat{T}_5^{[i]}([3]P+[1]Q,[3]P+[1]Q) = (g^3, g^4)$. 
  
		\[
\left(\begin{array}{rrrrr}
0 & 4 & 1 & 1 & 4 \\
0 & 2 & 2 & 0 & 1 \\
0 & 0 & 3 & 4 & 3 \\
0 & 3 & 4 & 3 & 0 \\
0 & 1 & 0 & 2 & 2
\end{array}\right), \quad
\left(\begin{array}{rrrrr}
0 & 2 & 3 & 3 & 2 \\
0 & 1 & 1 & 0 & 3 \\
0 & 0 & 4 & 2 & 4 \\
0 & 4 & 2 & 4 & 0 \\
0 & 3 & 0 & 1 & 1
\end{array}\right).
\]
We can also read off, for example, that $\widehat{T}_5^{[i]}(P,P) = \widehat{T}_5^{[i]}([2]P+Q,[2]P+Q) = 1$.  Thus the matrices have zeroes on the first column and on the coordinates $(a,b)$ which are multiples of $(2,1)$ modulo $5$.  This is as dictated by Theorem~\ref{thm:order}, because $P \in E[2+i]$ and $[2]P+Q \in E[2-i]$, which implies $[i]P = [3]P$ and $[i]([2]P+Q) = [2]([2]P+Q)$.  In other words, the subgroups $E[2\pm i] \cong \mathcal{O}/(2 \pm i)\mathcal{O}$ are the eigenspaces for the action of $\mathcal{O}$ on $E[5] \cong \mathcal{O}/5\mathcal{O}$. 
\end{example}

\section{Recovering partial torsion image information}
\label{sec:partial}

Our first observation is that when $E[m]$ is a cyclic $\mathcal{O}$-module, the pairings recover \emph{partial} information about the action of a hidden oriented isogeny $\phi$ on $E[m]$.

\begin{theorem}
	\label{thm:Nlambda}
	Let $E$ and $E'$ be $\mathcal{O}$-oriented supersingular curves over $\overline{\FF}_p$ upon which we can efficiently compute the action of a generator $\tau$ for $\mathcal{O}$.  Assume that the discrete logarithm in $\mu_m$ is efficiently computable. Assume also that $E[m]$ is a cyclic $\mathcal{O}$-module, and that the hidden oriented isogeny $\phi: E \rightarrow E'$ has known degree coprime to $m$.  Suppose we are given $P$ and $P'$ such that $\mathcal{O}P = E[m]$ and $\mathcal{O}P' = E'[m]$.  Then we can efficiently recover $N(\lambda)$ modulo $m$ for $\lambda \in \mathcal{O}$ such that $\phi P = [\lambda] P'$.
\end{theorem}

\begin{proof}
	We have
    \[
\widehat{T}^{\tau}_{m}(P,P)^{\deg \phi}
= \widehat{T}^{\tau}_{m}(\phi P, \phi P)
= \widehat{T}^{\tau}_{m}(\lambda P', \lambda P')
= \widehat{T}^{\tau}_{m}(P',P')^{N(\lambda)}.
\]
Note that by Theorem~\ref{thm:order}, $\widehat{T}^{\tau}_m(P,P)$ has order $m$.
Using the reduced pairing, we can solve a discrete log problem in $\mu_m$ to obtain $N(\lambda)$ modulo $m$.  \qed
\end{proof}

\begin{remark}
	\label{rem:sqrt} 
	This result improves upon a na\"ive exhaustive search over the possible images of a general point $P$ (on account of Theorem~\ref{thm:order}, we cannot use an eigenvector).  More precisely, one could attack the class group action problem by trying all possible image points $\phi P$ for $P$, infer $\phi [\tau]P = [\tau] \phi P$, and use the imputed image of $E[m]$ for the SIDH attacks, checking for success at each attempt.  This is similar to \cite[Section 4.1]{MSIDH}, for example.  Here, the knowledge of $N(\lambda)$ restricts $\phi P$ to typically around $m$ possible images (between $m\prod_{\text{prime } q \mid m}(1-1/q)$ and $m\prod_{\text{prime } q \mid m}(1+1/q)$), rather than all $m^2$.  To run such an attack, we need the degree of $\phi$ to be known and $m^2 > \deg \phi$, $m$ to be coprime to $\deg \phi$, and $m$ to be smooth.  Since we have great freedom in choosing $m$, we can expect to choose an $m$ around $\sqrt{\deg \phi}$. 

		The example of \cite{attacks} described by \eqref{eqn:weilattack} in Section 2.3 shows that when $m$ is a power of a prime $\ell$ that splits in $\QQ(\sqrt{-p})$, the classical $m$-Weil pairing also provides an attack with this runtime. However, with the sesquilinear pairing, one does not require splitting conditions. 
\end{remark}

\begin{remark}
	\label{rem:decisional}
	This and other similar results in this paper and in \cite{Level} are a caution against Decisional Diffie-Hellman problems in which one must decide if a given point is the image point of a specified torsion point under a hidden isogeny.  A result like the previous one reduces the possibilities for the torsion image (without pinning it down entirely).  For an example, the IND-CPA hardness of SiGamal \cite{SiGamal} depends upon such a problem, called the P-CSSDDH assumption.  This is discussed in \cite[Section 6.1]{attacks}, where the authors lament the triviality of the available self-pairings.  There are non-trivial pairings of the type $\widehat{T}$ which would apply to the SiGamal situation, but only if we had access to a different orientation on the curves and isogeny.  The Frobenius orientation used in the P-CSSDDH assumption results in a trivial pairing once again, because the torsion is contained in the base field.  
\end{remark}

\begin{remark}
	There is a sense in which we cannot hope to obtain more information than $N(\lambda)$ modulo $m$ using these methods. If we post-compose our isogeny with an endomorphism from $\mathcal{O}$ of norm $1$ modulo $m$, then we do not change the degree modulo $m$, but we do change $\lambda$, replacing it with another $\lambda'$ having the same norm modulo $m$.  To detect the difference, we must feed in more information than just the degree modulo $m$.
	In fact, it is possible to recover the same result by a different method.  Take a basis for $E$ and $E'$ and change basis so that the Weil pairing takes a canonical diagonal form.  Then the set of possible endomorphisms in $\mathcal{O}$ that preserve this diagonal form turns out to be the same `degree of freedom' of $\lambda$ observed above.  
	The pairings from \cite{attacks} can be seen as getting around this by assuming $\lambda \in \ZZ$, in which case $N(\lambda)$ pins down $\lambda$ more effectively.
\end{remark}

\begin{remark}
	In principle, the result above doesn't require using $\widehat{T}$; it could be phrased in terms of one of the coordinates in Theorem~\ref{thm:tate-sesqui} \eqref{item:tn}.  This wouldn't violate the classification of cyclic self-pairings in \cite{attacks} because the domain is not $\ZZ$-cyclic.
\end{remark}

\section{Reduction from SIDH$_1$ to SIDH}
\label{sec:partial2}

In \cite{Level}, the authors consider a variety of variants on the SIDH problem which can be parameterized by level structure for on the $m$-torsion preserved by $\phi$. In particular, they define the following problem.

\begin{problem}[SIDH$_1$]
Fix $d, m \in \ZZ$.  Let $E, E'$ be elliptic curves defined over $\FF_q$, where $m$ is coprime to $q$.  Let $P \in E[m]$ have order $m$.  Suppose there exists an isogeny $\phi: E \rightarrow E'$ of known degree $d$ and $\phi P$ is given.  Find $\phi$.
\end{problem}

This can be compared to the classical SIDH problem, in which we are given full torsion image information.

\begin{problem}[SIDH]
	Fix $d, m \in \ZZ$.  Let $E, E'$ be elliptic curves defined over $\FF_q$, where $m$ is coprime to $q$.  Let $P, Q$ form a basis for $E[m]$.  Suppose there exists an isogeny $\phi: E \rightarrow E'$ of known degree $d$ and $\phi P$ and $\phi Q$ are given.  Find $\phi$.
\end{problem}

In either case we refer to $m$ as the \emph{level} of the SIDH or SIDH$_1$ problem.  The authors of \cite{Level} show that if $m$ has a large smooth square factor, then SIDH$_1$ of level $m$ (a single torsion point image of order $m$) reduces to SIDH of level $O(\sqrt{m})$ (two torsion point images of order $O(\sqrt{m})$).  More recently, a manuscript in preparation (presented at Caipi Symposium 2024 \cite{unpubCastryck})  generalizes the SIDH attacks of \cite{SIDHbreak1,SIDHbreak2,SIDHbreak3}, directly attacking SIDH$_1$ without the requirement that $m$ have a large square factor.  Both approaches require that $m > \deg \phi$.

Here we show that, \emph{if we have an oriented isogeny}, knowing a single image of order $m$ is enough to reduce to SIDH of level $m$ (on the same curve), assuming only that $m$ is smooth, with no assumption on $m$ being square, and no loss in level.  Thus using the SIDH attacks requires only $m^2 > \deg \phi$. 

Although the proof relies on taking an `imaginary quadratic viewpoint,' it does not make use of the sesquilinear pairings.

\begin{theorem}
	\label{thm:sidh}
    Let $E$ and $E'$ be $\mathcal{O}$-oriented supersingular curves over $\overline{\FF}_p$, upon which we can efficiently compute the action of endomorphisms from $\mathcal{O}$.  Assume that $m$ is smooth and coprime to the discriminant.  Assume also that $E[m]$ is a cyclic $\mathcal{O}$-module, and that the hidden isogeny $\phi: E \rightarrow E'$ has known degree coprime to $m$ and is compatible with the $\mathcal{O}$-orientations. 
    Then the problem SIDH$_1$ of level $m$ to find $\phi$ reduces, in a polynomial number of operations in the field of definition of $E[m]$, to SIDH of level $m$ on the same curve $E$ and same $\phi$.
\end{theorem}

\begin{proof} For convenience, write $\mathcal{O} = \ZZ[\tau]$.
We are given $\phi R$ for some point $R \in E[m]$ of order $m$.  We wish to recover a second torsion point image resulting in an SIDH problem.  First, by Sunzi's Theorem, we can reduce the problem to prime powers $m=q^k$.  By assumption, $q$ is not ramified.  Hence we may assume $q$ is split or inert. 

\textbf{Case that $q$ is inert.} We know $\mathcal{O}R$ is an $\mathcal{O}$-submodule of $E[m] \cong \mathcal{O}/m\mathcal{O}$.  If $q$ is inert, it must be isomorphic to $\mathcal{O}/q^s\mathcal{O}$.  However, $\mathcal{O}/q^s\mathcal{O}$ doesn't have elements of additive order $q^k$ unless $s = k$.  Thus $\mathcal{O} R = E[m]$. 
Given any other point $Q$, we may compute $\eta$ such that $Q = [\eta] R$ (using basis $R$ and $[\tau]R$).  Then $\phi Q = \phi[\eta]R = [\eta] \phi R$.

\textbf{Case that $q$ is split.} 
 Write $m = q^k = \mathfrak{b} \overline{\mathfrak{b}}$, where $N(\mathfrak{b}) = m$. Write $\ker \mathfrak{b} := \{ P \in E[m] : \beta P = \mathcal{O} \text{ for all } \beta \in \mathfrak{b}\}$, and similarly for $\overline{\mathfrak{b}}$.  Then these are distinct cyclic subgroups of order $m$.  Thus there exists a basis $S, T$ for $E[m]$ so that $T \in \ker \mathfrak{b}$ and $S \in \ker \overline{\mathfrak{b}}$.  Similarly, let $S'$ and $T'$ be a basis for $E'[m]$ so that $T' \in \ker \mathfrak{b}$ and $S' \in \ker \overline{\mathfrak{b}}$.  To find such subgroups, one can use linear algebra, as follows.  The problem of finding $\ker \mathfrak{b}$ can be rephrased as solving for coefficients $a$ and $b$ for $T = aP + bQ$ in terms of a basis $P,Q$ for $E[m]$, subject to linear conditions determined by the action of $\mathfrak{b}$, which we can make explicit in terms of the known action of $\tau$.  In addition, by adding a gcd condition on the coefficients, one can choose $T$ to be of full order $m$.

Now the mapping $\phi$, as a matrix from basis $S,T$ to basis $S',T'$, is diagonal, with some integers $k_1$ and $k_2$ on the diagonal (as $\phi$ respects the $\mathcal{O}$-orientation).  By writing $R$ and $\phi R$ in the relevant bases, namely $R = [a]S + [b]T$, $\phi R = [c] S' + [d] T'$, we learn that $ak_1 \equiv c, bk_2 \equiv d \pmod m$, where $a,b,c,d$ are known.  We also know that $\deg \phi \equiv k_1k_2 \pmod{m}$.  Without loss of generality, at least one of $a$ or $b$ is coprime to $m$, so we know at least one of $k_1$ or $k_2$, and the degree equation then gives us the other.
\qed
\end{proof}

\section{Diagonal SIDH}
\label{sec:diag}

The following problem arises in \cite[Lemma 6 and Section 5.6]{Level}.

\begin{problem}[Diagonal SIDH]
	Fix $d, m \in \ZZ$.  Let $E, E'$ be elliptic curves defined over $\FF_q$, where $m$ is coprime to $q$.  Let $P,Q \in E[m]$ form a basis.  Suppose there exists an isogeny $\phi: E \rightarrow E'$ of known degree $d$.  Suppose that generators $P'$ of $\langle \phi P \rangle$ and $Q'$ of $\langle \phi Q \rangle$ are known.  Find $\phi$.
\end{problem}

Interestingly, when the curves are oriented, the Diagonal SIDH problem is amenable to a pairing-based attack, at least for certain conditions on $E[m]$. 

\begin{theorem}
	\label{thm:festa}
	Suppose $E$ and $E'$ are $\mathcal{O}$-oriented (and one can compute the action of the endomorphisms efficiently, as usual).  Let $m > 4\deg \phi$ be a smooth integer such that modulo $m$, $1$ has polynomially many square roots. Then Diagonal SIDH with known degree for an oriented isogeny $\phi: E \rightarrow E'$ is solvable in polynomial time, provided $\mathcal{O}P = E[m]$ or $\mathcal{O}Q = E[m]$.
\end{theorem}

\begin{proof}
	Let the Diagonal SIDH problem be given in terms of basis $P, Q$ for $E[m]$ and generators $P'$ and $Q'$ for $\langle \phi P \rangle$ and $\langle \phi Q \rangle$ respectively.
	Assume without loss of generality that $\mathcal{O}P = E[m]$.  Then by Theorem~\ref{thm:Nlambda}, we can efficiently recover $N(\lambda)$ modulo $m$ such that $\phi P = [\lambda] P'$.  However, the Diagonal SIDH setup guarantees that $\lambda \in \ZZ$, hence we have recovered $\lambda^2$ modulo $m$. By assumption, this gives only polynomially many possible values for $\lambda$, each of which can be tested by running the SIDH attacks, until one recovers $\phi$.\qed
\end{proof}

An instance of the Diagonal SIDH problem is the problem underlying the FESTA cryptosystem \cite[Problem 7]{festa}.  In this case $m$ is chosen to be a power of $2$, so the attack above would apply if FESTA were instantiated in a situation where the isogeny was oriented (for known orientations).  Assuming an $\mathcal{O}$-orientation, the condition $\mathcal{O}P = E[m]$ or $\mathcal{O}Q = E[m]$ is reasonably likely to occur by chance if not explicitly avoided.

\begin{corollary}
    \label{cor:festa}
	The hard problem underlying FESTA, namely CIST (see \cite{festa}), with $m > 4\deg \phi$, reduces to finding explicit $\mathcal{O}$-orientations of the curves $E$ and $E'$ respected by the isogeny $\phi$.
\end{corollary}

\begin{remark}
    In \cite[Section 5.6]{Level}, it is shown how to reformulate the problem of finding an isogeny of fixed degree $d$ between oriented curves (the class group action problem) as a Diagonal SIDH problem, where $m$ is a product of 
    primes split in $\mathcal{O}$.
The method of reduction, in brief, uses the eigenspaces associated to a split prime in the orientation, which must map to each other. However, the conditions under which Theorem~\ref{thm:festa} applies -- that $m$ have few square roots, and $P$ or $Q$ be generators of $E[m]$ as an $\mathcal{O}$-module -- both fail in the Diagonal SIDH problems that result from the reduction of \cite{Level}.  This means we cannot chain these attacks together to attack class group action problems!
\end{remark}

\section{When $m$ divides the discriminant}
\label{sec:ramified}

Suppose $m \mid \Delta_\mathcal{O}$, where $m = N(\tau)$ for $\tau \in \mathcal{O}$.  In this case the pairing $\widehat{T}^\tau_m$ becomes trivial.  However, a modification is more interesting.  Let $m \in \ZZ$, and define
\begin{align*}
T'_m: E[m](\FF) \times & E(\FF)/[m]E(\FF) \rightarrow ((\FF^*)/(\FF^*)^m)^{\times 2}, \\
	T'_m(P,Q) &= \left( t_{m}([\tau]P,Q), t_{m}(P,Q) \right).
\end{align*}
This modification does not preserve all of the properties of Theorem~\ref{thm:tate-sesqui} but importantly, it is bilinear and inherits compatibility from $t_{m}$, so that for $\phi : E \rightarrow E'$ compatible with $\mathcal{O}$, we have
\[
	T'_m(\phi Q, \phi Q) = T'_m(Q,Q)^{\deg \phi}.
\]

In \cite{attacks}, the authors use generalized pairings to determine the image of a single torsion point in $E[m]$, and then reduce to SIDH with $E[\sqrt{m}]$ when $m$ is a smooth square.  As previously mentioned, recent further development of the SIDH attacks (in preparation \cite{unpubCastryck}) generalize to image information on subgroups of a large enough size, not just full torsion subgroups, which effectively removes the restriction that $m$ be square.

Inspired by this result, we develop a similar reduction using the pairing above.  The main advantage of our situation over that in \cite{attacks} is the computation of the pairing, which requires only operations in the field of definition of $E[m]$.  Because the pairings used in \cite{attacks} may require a move to the field of definition of $E[m^2]$, our pairings result in a speedup in cases where that field of definition is large.

    \begin{proposition}
	    \label{prop:order-t}
	    Let $E$ be an elliptic curve oriented by $\mathcal{O} = \ZZ[\tau]$.  
	    Let $\FF$ be a finite field containing the $m$-th roots of unity.  Suppose $E[m] = E[m](\FF)$ is a cyclic $\mathcal{O}$-module.
	    Let $P \in E[m]$ have order $m$.  
	    Then the multiplicative order $T'_m(P,P)$ is at least $m/t$ where $t$ is the minimal positive integer such that $[t]E[m] \subseteq \mathcal{O}P$.
    \end{proposition}

\begin{proof}
	The classical Tate-Lichtenbaum pairing
	\[
		t_{m}: E[m](\FF) \times E(\FF)/[m]E(\FF) \rightarrow \FF^* / (\FF^*)^{m}
	\]
	is non-degenerate and, for $P$ of order $m$, there exists a $Q$ so $t_{m}(P,Q)$ has order $m$ (Proposition~\ref{prop:tate}).  
	Let $t$ be the minimal positive integer for which $[t]E[m] \subseteq \mathcal{O}P$. Then $[t]Q = [a + \tau b]P$.   Using Proposition~\ref{prop:tate},
	\begin{align*}
		T'_m(P,Q)^t &= T'_m(P,[a+\tau b]P) \\
			    &= T'_m(P,P)^{a}T'_m(P,[\tau] P)^{b} \\
			    &= \left(t_m([\tau]P,P)^{a}t_m([\tau]P,[\tau]P)^{b},t_m(P,P)^{a}t_m(P,[\tau]P)^{b}\right) \\
        &= \left(t_m([\tau]P,P)^{a}t_m(P,P)^{N(\tau)b},t_m(P,P)^{a + Tr(\tau)b}t_m([\tau]P,P)^{-b}\right).
	\end{align*}

	The left side has order $m/t$ by Proposition~\ref{prop:tate}.  Thus the right side has order $m/t$.  This is the image of $T'_m(P,P)$ via a linear transformation of determinant $N(\tau)b^2 + a^2 + Tr(\tau)ab = N(a + \tau b)$.  
 Therefore $T'_m(P,P)$ must have order at least $m/t$.\qed
\end{proof}

In the following, we assume $E[m]$ is a cyclic $\mathcal{O}$-module.  By Theorem~\ref{thm:cyclic}, it suffices that the $\mathcal{O}$-orientation be $m$-primitive.

\begin{theorem}
	\label{thm:ramified}
	Let $E$ and $E'$ be $\mathcal{O}$-oriented elliptic curves.
	Suppose there exists an oriented isogeny $\phi: E \rightarrow E'$ of known degree $d$.  
	Let $m$ be smooth, coprime to $d$, and chosen so that there are only polynomially many square roots of $1$ modulo $m$.  
 Suppose $m \mid \Delta_\mathcal{O}$.
	Suppose that $E[m]$ is a cyclic $\mathcal{O}$-module.
	Suppose $P \in E[m]$ such that $\mathcal{O}P = E[m]$, and $P' \in E'[m]$ such that $\mathcal{O}P' = E'[m]$. Then there exists an efficiently computable point $Q \in E[m]$ of order $m$ such that a subset $S \subset E'[m]$ of polynomial size containing $\phi(Q)$ can be computed in polynomially many operations in the field of definition of $E[m]$.
\end{theorem}

\begin{proof}

	Choose a point $P \in E[m]$ such that $\mathcal{O} P = E[m]$.
	Choose a point $P' \in E'[m]$ such that $\mathcal{O} P' = E'[m]$. Then $T'_{m}(P,P)$ and $T'_{m}(P',P')$ have order $m$ by Proposition~\ref{prop:order-t}.  Then
  \[
	  {T}'_m(P,P)^{\deg \phi}
	  = T'_m( \phi P,\phi P)
	  = T'_m( \lambda  P',\lambda P')
	  = T'_m( P',P')^{N(\lambda)}.
  \]
  (Note that $\lambda$ is an endomorphism, so here we use compatibility with isogenies, not sesquilinearity in general.) 
  Using a discrete logarithm, we can compute $N(\lambda) \pmod m$. 

Observe that the definition of $T_m'$ actually depends only on $\tau$ modulo $m$, and hence on $\ZZ[\tau]$ modulo $m$.  So we now choose $\tau$ in a specific way, possibly not generating all of $\mathcal{O}$ but only generating it modulo $m$.  In short, we claim the existence of $\tau \in \mathcal{O}$ with certain properties, namely that 

\begin{enumerate}
    \item $\ZZ[\tau] \equiv \mathcal{O}$ modulo $m$;
    \item $Tr(\tau)\equiv N(\tau) \equiv 0 \pmod{m'}$ where $m' = m/4$ if $4 \mid m$; $m' = m/2$ if $m \equiv 2 \pmod{4}$; and $m' = m$ otherwise.
\end{enumerate}
The existence of such a $\tau$ is a consequence of $m \mid \Delta_\mathcal{O}$, as follows. 
A generator for $\mathcal{O}$ is given by $\sigma = \frac{\Delta + \sqrt{\Delta}}{2}$ having trace $\Delta$ and norm $\frac{1}{4}(\Delta - \Delta^2)$. Then $\tau = 2\sigma$ already has the required properties if $m$ is odd.  If $m$ is even, then $4 \mid \Delta$, and the norm is divisible by $m'$, so $\tau = \sigma$ suffices.  
 Hence the minimal polynomial of $\tau$ is $x^2$ modulo $m'$, and $[\tau^2] E'[m] \subseteq E'[m/m'] \subseteq E'[4]$. 
  
  Write $\lambda \equiv a + b\tau$ modulo $m$. Then $N(\lambda) \equiv a^2 \pmod{m'}$.  
  Since the factorization of $m'$ is known, by assumption, we have an efficiently computable set of polynomial size of possible values of $a$.
Compute $[a][\tau]P'$.  For the correct $a$, this is the image of $[\tau]P$ under $\phi$ up to addition of a $4$-torsion point, since
  \[
	  \phi [\tau]P = [\lambda] [\tau] P' \in [a][\tau]P' + E'[4].
	  \]
	  Trying all possible values of $a$, and setting $Q = [\tau]P$, we obtain the set 
   \[
   \{ [a][\tau]P' : a^2 \equiv N(\lambda) \pmod m \} + E'[4]
   \]
   required by the statement. Observe that $P,Q$ form a basis for $E[m]$ by construction, so $Q$ has order $m$. \qed
\end{proof}

For each possible value of $a$, we have a putative $\phi([\tau]P)$, i.e. the action of $\phi$ on a single $m$-torsion point. The results of \cite{attacks,Level} can now be applied if $m$ is a smooth square, $d$ is powersmooth and $m > 4d$ , to reduce to  SIDH.  Alternatively, loosening the restriction that $m$ is a square will be possible with the new generalizations of SIDH mentioned above \cite{unpubCastryck}.

The following example demonstrates a new growing family of parameters for which solving the class group action problem (with known degree) is polynomial instead of exponential using Theorem~\ref{thm:ramified}.  

\begin{example}
	\label{ex:wouter}
	This is based on an example communicated to the authors by Wouter Castryck.
	Let $E: y^2 = x^3 + x$.  Let $p$ be a prime of the form $4 \cdot 3^r - 1$ with $r > 0$.  This curve is supersingular with $j=1728$ and endomorphisms $[i]: (x,y) \mapsto (-x, iy)$ and $\pi_p: (x,y) \mapsto (x^p,y^p)$.  Let
	\[
		\tau := \frac{i+\pi_p}{2} \in \End(E).
	\]
	Then $\tau^2 = - \frac{p+1}{4} = -3^{r}$, so $N(\tau) = 3^{r}$ and $Tr(\tau) = 0$.  Let $\mathcal{O} = \ZZ[\tau]$, having $N(\tau)\mid \Delta_\mathcal{O}$.  Let $m=3^r$.  Then $m \mid \Delta_\mathcal{O}$. 
 
 Since $\pi_p^2 = [-p]$, $E(\FF_{p^2}) \cong (\ZZ/(p+1)\ZZ)^2 \cong (\ZZ/4\cdot3^r\ZZ)^2$ \cite[Ex. 5.16.d]{silv09}. Therefore $E[3^r] \subseteq E(\FF_{p^2})$.

 Let $Q$ be an $\mathcal{O}$-generator of $E[3^r]$.
	Then by Proposition~\ref{prop:order-t}, $T'_{m}(Q,Q)$ has order $3^r$, and the polynomially many operations to run the attack of Theorem~\ref{thm:ramified} take place in $\FF_{p^2}$.  The SIDH portion of the attack requires that $4d < m = 3^{r}$.  

	By contrast, using the methods of \cite[Section 6.1]{attacks}, one would need a generalized pairing value for which only methods of computation taking place in the field of definition of $E[3^{2r}]$ are known, and this field degree grows exponentially with $r$.  (Specifically, in \cite[Section 5, p. 20]{attacks}, the authors give an estimated runtime for the pairing needed in the attack, noting that their method requires dividing a point by $m$ and working in the resulting field extension of degree as much as $O(m^2)$.)  That means that what was an exponential runtime in terms of $r$ under \cite{attacks} becomes a polynomial one using Theorem~\ref{thm:ramified}. 
\end{example}

\section{Supersingular class group action in the presence of another orientation}
\label{sec:twoorient}

The following theorem shows that, if we have two distinct orientations respected by $\phi$, then we can recover the action of $\phi$.  

Suppose $\mathcal{O} \subseteq \End(E)$. 
We use the notation $\mathcal{O}^\perp$ for the quadratic order orthogonal to $\mathcal{O}$ within the endomorphism ring, with respect to the geometry induced by the quaternion norm.

\begin{theorem}
	Let $E$ and $E'$ be supersingular elliptic curves for both of which we know orientations by two quadratic orders $\mathcal{O}$ and $\mathcal{O}'$ which together generate a rank $4$ sub-order of the endomorphism ring.  Let $\phi : E \rightarrow E'$ be an isogeny of known degree $d$.  Let $m$ be smooth, coprime to the discriminants of $\mathcal{O}$ and $\mathcal{O}'$, and suppose $1$ has only polynomially many square roots modulo $m$.  Suppose $\phi$ respects both the $\mathcal{O}$ and $\mathcal{O}'$ orientations.  Suppose $\mathcal{O}$-module generators are known for both $E[m]$ and $E'[m]$.  Suppose, finally, that $\mathcal{O}^\perp$ has elements of norm coprime to $m$.  Let $P \in E[m]$.  Then a subset of $E'[m]$ of polynomial size containing $\phi P$ can be computed in a polynomial number of operations in the field of definition of $E[m]$. 
\end{theorem}

\begin{proof}

 Suppose $E[m] = \mathcal{O}P$ and $E'[m] = \mathcal{O}P'$.
	
	Let $\sigma \in \End(E)$ be chosen to have norm coprime to $m$. 
	Write $\lambda^\sigma$ for an element which participates in the equivalence $\lambda^\sigma \sigma \equiv \sigma \lambda \pmod m$.  Suppose $\lambda^\sigma \in \mathcal{O}$. 
	Then
	\begin{align*}
\widehat{T}_m^{\tau}([\sigma] P, P)^{\deg \phi} 
&= \widehat{T}_m^{\tau}( \phi [\sigma] P, \phi  P ) \\
&= \widehat{T}_m^{\tau}( [\sigma] \phi P,  \phi P) \\
&= \widehat{T}_m^{\tau}( [\sigma] [\lambda] P',  [\lambda] P' ) \\
&= \widehat{T}_m^{\tau}( [\lambda^\sigma] [\sigma] P', [\lambda] P' ) \\
&= \widehat{T}_m^{\tau}( [\sigma] P', P')^{\overline{\lambda^\sigma} \lambda}.
\end{align*}
Since the norm of $\sigma$ is coprime to $m$, $\widehat{T}_m^{\tau}([\sigma] P, P)$ has the same order as $\widehat{T}_m^{\tau}(P,P)$, which is $m$ by Theorem~\ref{thm:order}. 
Thus, we can compute $\overline{\lambda^\sigma}\lambda$ modulo $m$ by performing a discrete logarithm in $\mu_m$.   

We will now apply the above for two specially chosen $\sigma \in \End(E)$.  Since we can compute the action of $\mathcal{O}$ and $\mathcal{O}'$, we can compute the action of anything they generate.  Thus, we choose $\sigma_1 =1$ (so $\lambda^{\sigma_1} = \lambda$), and then some $\sigma_2 \in \mathcal{O}^\perp$, so $\lambda^{\sigma_2} = \overline{\lambda}$.  Assuming that $\mathcal{O}^\perp$ contains elements of norm coprime to $m$, from this, we obtain both $N(\lambda)$ and $\lambda^2$ modulo $m$.

Using $N(\lambda)$ and $\lambda^2$, we can solve for polynomially many possibilities for $\lambda$ modulo $m$ (this requires smoothness, so $m$ can be factored).  Then we have obtained $\phi P$.  From this we can compute any other $\phi R$ by solving for $R = [\mu] P$ and observing that $\phi R = \phi [\mu] P = [\mu] \phi P$.\qed 
\end{proof}

\subsection*{\ackname}  The authors are grateful to Wouter Castryck, Steven Galbraith, Marc Houben, Massimo Ostuzzi, Lorenz Panny, James Rickards and Damien Robert for helpful discussions.  They are also grateful to the anonymous referees for feedback which much improved the paper.  Both authors have been supported by NSF CAREER CNS-1652238 and NSF DMS-2401580 (PI K. E. Stange).

\subsection*{\discintname} The authors have no competing interests to declare that are relevant to the content of this article.

\bibliographystyle{alpha}
%\bibliography{refs}

\newcommand{\etalchar}[1]{$^{#1}$}

\end{document}